\documentclass[sn-mathphys-num]{sn-jnl}% Math and Physical Sciences Numbered Reference Style
\usepackage{graphicx}%
\usepackage{multirow}%
\usepackage{amsmath,amssymb,amsfonts}%
\usepackage{amsthm}%
\usepackage{mathrsfs}%
\usepackage[title]{appendix}%
\usepackage{xcolor}%
\usepackage{textcomp}%
\usepackage{manyfoot}%
\usepackage{booktabs}%
\usepackage{algorithm}%
\usepackage{algorithmicx}%
\usepackage{algpseudocode}%
\usepackage{listings}%
\usepackage{arydshln}
\theoremstyle{thmstyleone}%
\newtheorem{theorem}{Theorem}%  meant for continuous numbers
\newtheorem{proposition}[theorem]{Proposition}%

\newtheorem{lemma}{Lemma}[section]
\newtheorem{assumption}{Assumption}[section]

\theoremstyle{thmstyletwo}%
\newtheorem{remark}{Remark}%

\theoremstyle{thmstylethree}%
\newtheorem{definition}{Definition}%

\begin{document}

\title[Article Title ]{Inexact Regularized Quasi-Newton Algorithm for Solving Monotone Variational Inequality Problems}

%%=============================================================%%
%% GivenName	-> \fnm{Joergen W.}
%% Particle	-> \spfx{van der} -> surname prefix
%% FamilyName	-> \sur{Ploeg}
%% Suffix	-> \sfx{IV}
%% \author*[1,2]{\fnm{Joergen W.} \spfx{van der} \sur{Ploeg}
%%  \sfx{IV}}\email{iauthor@gmail.com}
%%=============================================================%%

\author[1]{\fnm{} \sur{Yuge Ye}}\email{fujianyyg@163.com}
%\equalcont{These authors contributed equally to this work.}

\author*[1]{\fnm{} \sur{Qingna Li}}\email{qnl@bit.edu.cn}

\author[2]{\fnm{} \sur{Deren Han}}\email{handr@buaa.edu.cn}

%\author[2]{\fnm{Third} \sur{Deren Han}}\email{handr@buaa.edu.cn}
%\equalcont{These authors contributed equally to this work.}

\affil[1]{\orgdiv{Department of Mathematics and Statistics}, \orgname{Beijing Institute of Technology}, \orgaddress{\street{No.5 Yard, Zhong Guan Cun South Street}, \city{Beijing}, \postcode{100081}, \state{Beijing}, \country{China}}}

\affil[2]{\orgdiv{Department of Mathematical Science}, \orgname{Beihang University}, \orgaddress{\street{No.37 Xueyuan Road}, \city{Beijing}, \postcode{100191}, \state{Beijing}, \country{China}}}

%\affil[1]{\orgdiv{Department of Mathematics and Statistics}, \orgname{Beijing Institute of Technology}, \orgaddress{\street{No.5 Yard, Zhong Guan Cun South Street}, \city{Beijing}, \postcode{100081}, \state{Beijing}, \country{China}}}

%\affil[2]{\orgdiv{Department of mathematical science}, \orgname{Beihang University}, \orgaddress{\street{No.37 Xueyuan Road}, \city{Beijing}, \postcode{100191}, \state{Beijing}, \country{China}}}

%%==================================%%
%% Sample for unstructured abstract %%
%%==================================%%

\abstract{Newton's method has been an important approach for solving variational inequalities, quasi-Newton method is a good alternative choice to save computational cost. In this paper, we propose a new method for solving monotone variational inequalities where we introduce a merit function based on the merit function. With the help of the merit function, we can locally accepts unit step size. And a globalization technique based on the hyperplane is applied to the method. The proposed method applied to monotone variational inequality problems is globally convergent in the sense that subproblems always have unique solutions, and the whole sequence of iterates converges to a solution of the problem without any regularity assumptions.  We also provide extensive numerical results to demonstrate the efficiency of the proposed algorithm.}


%%\pacs[JEL Classification]{D8, H51}

%%\pacs[MSC Classification]{35A01, 65L10, 65L12, 65L20, 65L70}

\maketitle

\section{Introduction}\label{sec1}
\numberwithin{equation}{section}
Let $C$ be a non-empty closed convex subset of the Euclidean distance space $\mathbb{R}^n$, and let $F: \mathbb{R}^n \to \mathbb{R}^n$ be a continuous mapping. The variational inequality problem (denoted as VIP(F,C) is to find $x^* \in C$ such that:
\begin{equation}\label{VIP}
 \langle F(x^*), x-x^* \rangle \ge 0, \quad \forall \, x \in C,
\end{equation}
where $\langle \cdot , \cdot \rangle$ denotes the inner product in $\mathbb{R}^n$.

Variational inequality problems provide a unified framework for various problems. For more detailed information on these problems, one may refer to \cite{Pang1990}. Newton's method has been a popular approach to solve \eqref{VIP} \cite{marcotte1987note, TajiA1993, peng1999hybrid,peng1999box, ferris1999feasible}. To improve Newton's method, algorithms proposed by \cite{marcotte1987note, TajiA1993, peng1999hybrid,peng1999box, ferris1999feasible} utilize merit-function and line search techniques to achieve global convergence.

There are some methods use hyperplane projection techniques to obtain a globally convergent Newton's method \cite{solodov1998globally,solodov1999new,solodov2000truly}. Based on such observation, Han et al. \cite{Han2004} proposed a feasible Newton's method to solve mixed complementarity problems. In \cite{Qu2020} and  \cite{Qu2017}, Qu proposed a projection-based Newton's method to solve variational inequality problems and split feasibility problems.  Fatemeh\cite{abdi2019globally} utilized hyperplane projection techniques and proposed a globally convergent BFGS method to solve pseudomonotone variational inequality problems. Theses method utilizes hyperplane projection techniques to ensure the global convergence of the algorithm. Although it does not require strong function conditions, it needs projection at each iteration, even when the convergence property of the unit step is good.

Motivated by the above observations in this paper, we propose an inexact regularized quasi-Newton's method(IRQN) to solve the monotone VIP \eqref{VIP}.
We utilize a merit function to enable the algorithm to automatically use unit step sizes in a certain local neighborhood of the solution. Our numerical results demonstrate the efficiency of our method.

The organization of the paper is as follows. In section 2, we introduce some preliminaries for VIP \eqref{VIP}. In section 3, we propose the so-called IRQN for VIP \eqref{VIP}. In section 4 and section 5, we analyze the global and local convergence results respectively. We conduct various numerical experiments in section 6 to verify the efficiency of the proposed method. Final conclusions are given in section 7.

Notations. $\| \cdot \|$ represents the $l_2$ norm, where $x \in \mathbb{R}^n$. $P_C(x)$ denotes the projection of $x$ onto the set $C$, that is, $P_C(x)=\operatorname{arg}\min\limits_{y\in C}\frac{1}{2}\|y-x\|$.

\section{Preliminaries}
\numberwithin{equation}{section}
\label{sec:manu}

\numberwithin{equation}{section}
In this section, we provide some basic concepts in VIP \eqref{VIP} and their related properties.

\begin{definition}
The normal cone $N_C(x)$ at $x\in C$ is defined as:
$$N_C(x)=
\begin{cases}
\left \{z:(y-x)^{\top}z\le 0,\  \forall \  y\in C\right \},& \text{if $x\in C$},\\
\emptyset , & \text{otherwise}.
\end{cases}$$
\end{definition}

\begin{proposition}\label{pro-2-1}\rm{\cite[p.165]{Pang1990} }
$x^*$ is a solution to VIP$(F,C)$ if and only if
\begin{align*}
0\in F(x^*) + N_C(x^*):= (F+N_C)(x^*).
\end{align*}
\end{proposition}

\begin{definition}
    For a mapping $F:\mathbb{R}^n\to \mathbb{R}^n$,
\begin{description}
\item[(i)] $F$ is monotone on $C$ if $\langle F(x)-F(y), x-y\rangle \ge 0$ for all $x, y\in C$;
\item[(ii)] $F$ is strictly monotone on $C$ if $\langle F(x)-F(y), x-y\rangle > 0$ for all $x, y\in C$, $x\neq y$;
\item[(iii)] $F$ is strongly monotone on $C$ with constant $\mu>0$ if $\langle F(x)-F(y),x-y\rangle \geq \mu\|x-y\|^2$ for all $x, y\in C$.
\end{description}
\end{definition}

\begin{proposition}\label{p1}\rm{\cite[p.267]{Bertsekas1989}}
Let $\alpha > 0$. Then $x^*$ is a solution to the VIP \eqref{VIP} if and only if
$$x^* = P_{C}(x^{*} - \alpha^{-1}F(x^{*})).$$
That is, $x^*$ is a fixed point of the mapping $H_{\alpha}(\cdot) : \mathbb{R}^n \to \mathbb{R}^n$, where $H_{\alpha}(x) = P_{C}(x - \alpha^{-1}F(x))$.
\end{proposition}
The projection mapping $P_{C}(\cdot)$ is nonexpansive, that is,
\begin{equation*}
\| P_{C}(x) - P_{C}(y)\| \le \|x-y\| , \  \forall \  x, y \in \mathbb{R}^n.
\end{equation*}

For a given constant $\alpha>0$, we consider the following merit function \cite{Fukushima1992} where $f_{\alpha} :\mathbb{R}^n \to \mathbb{R}$ is defined by
\begin{align}\label{merit}
f_{\alpha}(x) &= \max \left\{  - \langle F(x), y - x \rangle - \frac{1}{2}\langle y - x, \alpha(y - x)\rangle \mid y \in C \right\} \nonumber\\
&=-\langle F(x), H_{\alpha}(x)-x\rangle - \frac{\alpha}{2}\|H_{\alpha}(x)-x\| .
\end{align}

\begin{lemma}\label{lemma3} \rm{\cite[Theorem 3.1]{Fukushima1992}}
For $f_{\alpha}(\cdot)$ defined in \eqref{merit}, for any $ x\in C$, it holds that $f_{\alpha}(x) \ge 0$, and $f_{\alpha}(x) = 0$ if and only if $x$ is a solution to the VIP \eqref{VIP}. Thus, $x$ is a solution to VIP \eqref{VIP} if and only if it is a solution to the following optimization problem:
\begin{align*}
\min \quad & f_{\alpha}(x) \nonumber\\
\mbox{s.t.}\quad
& x \in C.
\end{align*}
\end{lemma}

\begin{proposition}\label{p5}\rm{\cite[Proposition 3.4]{TajiA1993}}
Let $x^*$ be a solution to the variational inequality problem \eqref{VIP}. If $F$ is strongly monotone on $C$ with modulus $\mu$, then $f_{\alpha}$ satisfies:
\begin{equation*}
f_{\alpha}(x)\ge \left(\mu - \frac{\alpha}{2} \right)\| x-x^*\|^{2},  \  \forall \  x \in C .
\end{equation*}
If $\alpha < 2\mu$, then $f_{\alpha}(x) \to \infty$ as long as $x\in C$ and $ \| x \| \to \infty$ .
\end{proposition}

\section{Inexact Regularized Quasi-Newton Method}
\label{sec:temp}

\numberwithin{equation}{section}
In this section, we propose a globally convergent quasi-Newton method for solving monotone VIP\eqref{VIP}.

Assume that F is monotone. The idea to solve VIP\eqref{VIP} is that we adapt the framework proposed by \cite{solodov2000truly}. Moreover, the following modifications are added to algorithm. (i) We introduce a merit function to enable the algorithm to use unit steps in a local neighborhood of the solution. (ii) We also introduce BFGS quasi-Newton method to avoid calculation of the Jacobian matrix. The framework of the algorithm is given as follows.

\textbf{Step 1:} In each iteration, given $x_k$ and an approximate matrix $B_k$ for the Jacobian, we need to find a $z_k$ approximate by solving a linear variational inequality subproblem:
\begin{align}\label{subproblem}
\langle \varphi_k, x - z_k \rangle \ge 0, \  \forall x\in C, \tag{VIP$(\varphi_k, C)$}
\end{align}
where $\varphi_k = F(x_k + (B_k + \mu_k)(z_k - x_k)$, $\mu_k=\alpha\|x_k - P_{C}(x_k - F(x_k)/\alpha)\|$, $\alpha > 0$ is a real number. As the iteration proceeds and the sequence converges, the regularization term $\mu_k$ also converges to 0.
We solve \ref{subproblem} in an inexact way.
We use some quasi-Newton method to update $B_k$, for example, the BFGS update in \cite{2000LiDH}.
Here we need to ensure that $B_k$ remains semi-positive definite at each iteration.

Since $B_k+\mu_k I$ is always positive definite, the subproblem \ref{subproblem} always has a unique solution $z \in C$. Moreover, the linear subproblem \ref{subproblem} is always easier to solve than the original problem \eqref{VIP}.

\textbf{Step 2:} Construct a hyperplane and do the projection. Define the residual function by $e_k = z_k - P_{C}(z_k - \varphi(z_k)).$ We need the following result.
\begin{proposition}
Let $x^*$ be the solution of VIP \eqref{VIP}. Define
\begin{eqnarray*}
y_k = z_k - e_k,\ \upsilon_k := F(z_k - e_k) - \varphi_k(z_k) + e_k ,\ \varepsilon_k := -\upsilon_k - \mu_k(y_k - x_k).
\end{eqnarray*}If
\begin{align}\label{varepsilon}
\| \varepsilon_k \| \le \eta \mu_k \|y_k - x_k\|,\  \eta\in (0,1),
\end{align} then, the following holds
$$\langle \upsilon_k, x^* - y_k \rangle \le 0.$$
\end{proposition}
\begin{proof}
By the property of projection, it holds that
$$
\langle z_k - \varphi_k(z_k) - (z_k - e_k), x - (z_k - e_k)\rangle \le 0.
$$
Therefore, $\langle -\varphi_k(z_k) + e_k , x - (z_k - e_k)\rangle \le 0,\  \forall x\in \mathbb{R}^n.
$
Then
$$h_k := -\varphi_k(z_k) + e_k \in N_{C}(z_k - e_k).
$$
By the definition of $y_k$ and $\upsilon_k$, it holds that
$$\upsilon_k \in (F+N_{C})(z_k - e_k).$$ Then by Cauchy-Schwarz inequality we have:
\begin{align}\label{3.3}
\langle \upsilon_k , x_k - y_k\rangle &= \mu_k \|x_k - y_k\|^2 - \langle \varepsilon_k , x_k - y_k\rangle \nonumber \\
&\ge \mu_k\|x_k - y_k\|^2 - \| \varepsilon_k\|\|x_k - y_k\|^2 > 0 .
\end{align}
Define the hyperplane:
\begin{align}\label{hyperplane}
H_k := \{x | \  \langle \upsilon_k ,x - y_k\rangle = 0\}.
\end{align}
By the monotonicity of $F + N_{C}$ \cite{rockafellar1969convex}, we know that $\langle \upsilon_k, x^{*} - y_k \rangle \le 0$.
\end{proof}

\begin{remark}
It is clear that if \eqref{varepsilon} holds, then the hyperplane \eqref{hyperplane} can separate the current iteration point $x_k$ and the solution $\bar{x}$, and projecting $x_k$ onto this hyperplane can reduce the distance between the iteration point and the solution. Otherwise, we will introduce a line search strategy to find the next iteration.
\end{remark}

\textbf{Step 3:} We use a step size strategy based on merit function. If $z_k$ can make the merit function strictly decrease, we use $z_k$ directly. Otherwise, we project and correct $x_k$ to obtain a new iteration point that is closer to the solution of problem (\ref{VIP}).

For the solution $x^{*}$ of problem \eqref{VIP}, it is clear that $$\langle F(x^{*}), x^{*} - y_k \rangle \ge 0.$$
Therefore, by the monotonicity of the operator $F$, we have
$$
\langle F(x^{*}) , x^{*} - y_k \rangle \le 0.
$$
In this case, the hyperplane obtained is
$$H_k :=\{x\in \mathbb{R}^n | \langle F(y_k),x-y_k\rangle = 0\}.$$
Similarly, using hyperplane projection can reduce the distance between the iteration point and the solution. In this situation, we have
$$-\alpha_k \langle F(y_k), z_k - x_k \rangle = \langle F(y_k), x_k - y_k \rangle \ge 0 .$$ If $\alpha_k \le 1$, then $$\langle F(y_k), z_k - y_k \rangle = (1 - \alpha_k) \langle F(y_k), z_k - x_k \rangle \le 0.$$
From a geometric perspective, it is possible that $z_k$ is closer to $x^*$ than $y_k$. Therefore, we can use this property to construct an algorithm.

It is worth noting that the algorithm in \cite{solodov2000truly} requires projection correction for $x_k$ in each iteration. By using of our step size strategy, we can achieve automatic use of pure Newton step.

The details of the algorithm are given in Algorithm \ref{algorithm}.

\begin{algorithm}
\caption{Inexact Regularized Quasi-Newton Algorithm (IRQN)}\label{algorithm}
%(Algorithm name)
% \hspace*{0.02in}
\textbf{S0. } $x_0 \in C, \eta \in (0,1), \lambda \in (0,1), \beta \in (0,1), \gamma \in (0,1), \alpha>0, h>0, r>0,$
$ \epsilon>0, \mu_0>0, \rho_0 <1, B_0 \in \mathbb{R}^{n\times n}$, $\{\mu_k \}\downarrow 0$, $\{\rho_k \}\downarrow 0$

%(Algorithm)
\textbf{S1.}  Stop if $\alpha \|x_k - P_{C}(x_k - F(x_k)/\alpha)\| \le \epsilon$. Otherwise, let $\mu_k \to 0$ and $\rho_k \to 0$, let $e_k = z_k - P_{C}(z_k - \varphi_k(z_k))$, calculate an inexact solution $z_k$ of \ref{subproblem} by some current algorithm, such that
\begin{align}\label{algorithm-error1}
\|e_k\| \le \rho_k \mu_k \|z_k - x_k\|,\  \langle e_k, \varphi_k(z_k) + z_k - x_k \rangle \le \rho_k \mu_k \|z_k - x_k\|^2.
\end{align}
If $z_k = x_k$, stop. Otherwise, go to \textbf{S2}.

\textbf{S2.} If
\begin{align} \label{newtonstep}
f_{\alpha}(z_k) \leq \gamma f_{\alpha}(x_k),
\end{align}
set $x_{k+1} = z_k$, otherwise go to \textbf{S3}.

\textbf{S3.}  Let $\upsilon_k:=F(z_k - e_k) - \varphi_k(z_k) + e_k,\  y_k:=z_k - e_k$, and $\varepsilon_k = -\upsilon_k - \mu_k(y_k - x_k)$. If
\begin{align}\label{algorithm-error3}
\|\varepsilon_k\| \le \eta \mu_k \|y_k - x_k\|,
\end{align}
go to \textbf{S4}, otherwise let $m_k$ be the smallest integer satisfying the following inequality:
\begin{align}\label{algorithm-linesearch}
\langle F(x_k + \beta^{m}(z_k - x_k)), x_k - z_k\rangle \geq \lambda(1-\rho_k)\|z_k-x_k\|^2.
\end{align}
Let $\alpha_k = \beta^{m_k}, y_k = x_k + \alpha_k (z_k - x_k)$, and $\upsilon_k = F(y_k)$. Go to \textbf{S4}.

\textbf{S4.}
\begin{align}\label{algorithm-project}
\hat{x}_k = x_k - \frac{\langle \upsilon_k,x_k - y_k\rangle}{\|\upsilon_k\|^2}\upsilon_k, \quad x_{k+1} = P_{C}(\hat{x}_k).
\end{align}

\textbf{S5.} Let $B_{k+1}$ updated by quasi-Newton method. Let $k = k+1$. Go to \textbf{S1}.
\end{algorithm}

\begin{remark}
The main difference between this algorithm and the algorithm in \cite{solodov2000truly} is the use of merit function (\ref{newtonstep}) and the quasi-Newton update. The main computational burden of the algorithm lies in solving the subproblems and line search.

It can be easily seen from \cite{solodov2000truly} that the line search (\ref{algorithm-linesearch}) in Algorithm \ref{algorithm} is well defined.
\end{remark}

\section{Global Convergence}
\numberwithin{equation}{section}
We need the following assumption.
\begin{assumption}\label{assumption-1}
The function $F$ is locally Lipschitz continuous in a neighborhood $X$ of $x^{*}$, i.e. $\|F(x) - F(y)\| \le L_1\|x-y\|, \  \forall x, y\in X,$
where $L_1$ is a positive constant.
\end{assumption}

The following theorem shows that Algorithm \ref{algorithm} is globally convergent.
\begin{theorem}\label{lemma-global}
Assume that $C$ is a nonempty, closed, and convex subset of $\mathbb{R}^n$, and $F$ is monotone and continuous on $C$. Assume that $B_k$ is positive definite for all $k>0$, and let $\{x_k\}$ be a sequence generated by Algorithm \ref{algorithm} starting from an arbitrary initial point. Then, $\{x_k\}$ is bounded. Let $r(x_k)=\|x_k - P_{C}(x_k-F(x_{k}))\|$. Let Assumption \ref{assumption-1} hold. Assume that there exist constants $C_1,C_2 > 0$ such that for some $k > k_0$,
\begin{align}\label{mu_k}
C_1 \ge \mu_k \ge C_2 \|r(x_k)\|.
\end{align}
If $\min{(1,1/C_1)} > \sup_{k}\rho_k$ and $ \rho_k \to 0$ , then $\{x_k\}$ converges to a solution $x^{*}$ of VIP(F,C).
\end{theorem}
\begin{proof}
Firstly, consider the case where \eqref{newtonstep} holds for infinitely many times, i.e., $f_{\alpha}(z_k) \le \gamma f_{\alpha}(x_k)$. Since $\gamma \in (0,1)$, it follows that $\lim\limits_{k\to \infty}f_{\alpha}(x_k) = 0.$ By Lemma \ref{lemma3} and the continuity of $f_{\alpha}(x)$, the iteration converges to a point $x^{*}$, which is a solution of VIP(F,C).
Thus, the result holds.

Now consider the case where \eqref{newtonstep} holds only finitely many times. In this case, the proof follows a similar method to that in \cite{solodov2000truly}, and we omit the details here.
\end{proof}
\begin{remark}
We can choose an appropriate parameter sequence, such as $\{ \mu_k \} = O (r_k)$, to satisfy \eqref{mu_k}. Similarly, the sequence $\{\rho_k\}$ can be chosen by requiring that the update of the quasi-Newton matrix ensures that $B_k$ remains positive semidefinite.
\end{remark}

\section{Local Convergence}
\numberwithin{equation}{section}
Next, we analyze the local convergence of the algorithm. We need the following results and assumptions.

\begin{lemma}\label{lemma-localineq} \rm {\cite[Theorem 3.1]{doi:10.1287/moor.12.3.474} }
Let $\hat{z}$ be a solution of the following linear variational inequality: $\hat{z}\in C, \  \forall x\in C,\  \langle Mz+q,x- \hat{z}\rangle \ge 0$.
Here, $M$ is a positive definite matrix, and $q$ is a vector. Then, for any $z\in C$, we have
$$\|z-\hat{z}\| \le \frac{1+\|M\|}{c_k}\|z - P_{C}(z-(Mz+q))\|,$$
where $c_k$ is the smallest eigenvalue of $(M+M^{\top})/2$.
\end{lemma}

\begin{lemma}\label{lemma-project}
Assume that $F$ is Lipschitz continuous on set $C$ with $\alpha >0$. Then, there exists a constant $C>0$ such that $\|x - P_{C}(x - F(x))\| \le C \|x - x^{*}\| $,
where $x^{*}$ is a solution of the VIP \eqref{VIP}.
\end{lemma}
\begin{proof}
\begin{align*}
\|x_k - P_{C}(x - F(x_k))\| &= \|x- P_{C}(x - F(x)) - x^* - P_{C}(x^* - F(x^*)) \| \\
&\le \|x - x^*\| + \|P_{C}(x - F(x)) - P_{C}(x^* - F(x^*))\| \\
&\le \|x - x^*\| + \|(x - F(x)) - (x^* - F(x^*))\| \\
&\le (2 + \alpha)\| x - x^* \|.
\end{align*}
Let $C = (2 + \alpha)$, the proof is finished.
\end{proof}

\begin{assumption}\label{assumption-2}
The matrix $\nabla F$ is Lipschitz continuous in a local neighborhood $X$ of $x^{*}$, i.e., $\|\nabla F(x) - \nabla F(y)\| \le L_2\|x-y\|, \  \forall \  x\in X,$
where $L_2$ is a positive constant.
\end{assumption}

\begin{assumption}\label{assumption-3}
There exists a constant $\mu>0$ in a local neighborhood $X$ of $x^{*}$ such that for any $d\in X$, it holds that $\langle \nabla F(x)d,d\rangle \ge \mu\|d\|^2$.
\end{assumption}

\begin{theorem}\label{lemma-localconverge}
Assume that $F$ is continuously differentiable and that Assumptions \ref{assumption-1}, \ref{assumption-2}, and \ref{assumption-3} hold in a local neighborhood of a solution $x^{*}$ of VIP(F,C). Let $\{x_k\}$ be generated by Algorithm \ref{algorithm} and let $\{z_k\}$ be the sequence of inexact solutions of \ref{subproblem} in Algorithm \ref{algorithm}, and let $\{\hat{z}_k\}$ be the sequence of accurate solutions of \ref{subproblem}. Assume that there exists a constant $D>0$ such that $\|B_k - \nabla F(x_k)\| \le D$. Furthermore, assume that conditions (i), (ii), (iii) hold, where

\begin{description}
\item[(i)] $\mu_k = O (r_k)$, $\mu_k \le C_1\|x_k - x^*\|$, $C_1 > 0$, $\ \rho_k \to 0$;
\item[(ii)] $\frac{1+\|B_k + \mu_kI\|}{c_k} \le C_2 $, where $c_k$ is the minimum eigenvalue of $B_k+\mu_kI$, $C_2 > 0$;
\item[(iii)] For sufficiently large $k$, $\|(B_k - \nabla F(x_k))(z_k - x_k)\| \le \delta_k \|z_k - x_k\|$ where $\{\delta_k\}$ is a positive sequence, and $\mu_{k} < \frac{1}{2}\mu$, $\delta_k \le \frac{1}{16}\mu$.
\end{description}
Then, for sufficiently large $k$, there exists a local neighborhood $X$ of $x^{*}$ such that for any $x_k\in X$, we have
\begin{align}\label{4.4}
\begin{cases}
\|\hat{z}_k - x^{*}\| \le \frac{L_2 + 14C_1 + 2D}{2\mu}\|x_k - x^{*}\|^2 + 16\delta_k\|x_k - x^{*}\|,\\
\|z_k - x^{*}\|\le \frac{L_2 + 12C_1 + 2D}{2\mu}\|x_k - x^{*}\|^2 + 16\delta_k\|x_k - x^{*}\|.
\end{cases}
\end{align}
\end{theorem}
\begin{proof}
We proceed with the following three steps.

\textbf{Step 1.} First, we derive an upper bound for $\|x^{*} - \hat{z}_k\|$. Let $X'$ be a local neighborhood of $x^{*}$ satisfying Assumption \ref{assumption-3} and assumptions (i)-(iii). Let $x\in D$ and let $\hat{z}_k$ be the solution of \ref{subproblem}, i.e.,
$$\langle F(x_k)+(B_k + \mu_k I)(\hat{z}_k - x_k),x - \hat{z}\rangle \ge 0, \  \forall x\in C.$$
By Assumption \ref{assumption-2}, it holds that
\begin{align}\label{super-1}
 F(x_k + u) - F(x_k) &= \int_{0}^{1}\nabla F(x_k + tu)u\,dt \nonumber \\
&=  \nabla F(x_k)u + \int_{0}^{1} (\nabla F(x_k + tu)-\nabla F(x_k))u\, dt \nonumber \\
&= \nabla F(x_k)u + R_k(u),
\end{align}
where
$$\|R_k(u)\| \le \int_{0}^{1}\|\nabla F(x_k + tu)-\nabla F(x_k)\|\|u\|\,dt \le L\int_{0}^{1}t\|u\|^2\,dt  = \frac{L_2}{2}\|u\|^2 .
$$
Combining with \eqref{super-1}, it gives that
\begin{flalign}\label{4.6}
&F(x^{*}) - \varphi_k(\hat{z}_k) \nonumber \\
= &F(x^{*}) - F(x_k)-(B_k + \mu_k I)(\hat{z}_k - x_k) \nonumber \\
= &F(x^{*})-F(x_k)-\nabla F(x_k)(x^{*}-x_k) +\nabla F(x_k)(x^{*}-x_k)  -(B_k + \mu_k I)(\hat{z}_k - x_k) \nonumber \\
= &R_k(x^{*}-x_k) + \nabla F(x_k)(x^{*}-\hat{z}_k +\hat{z}_k - x_k) -(B_k + \mu_k I)(\hat{z}_k - x_k) \nonumber \\
= &R_k(x^{*}-x_k) + \nabla F(x_k)(x^{*}-\hat{z}_k) - (B_k - \nabla F(x_k)+\mu_k I)(\hat{z}_k - x_k).
\end{flalign}

From the properties of the variational inequality, it holds that
$$\langle F(x^{*}),x^{*} - \hat{z}_k \rangle \le 0,\  -\langle \varphi_k(\hat{z}_k),x^{*}-\hat{z}_k\rangle \le 0 ,$$
which implies
\begin{align}\label{super-2}
\langle F(x^{*}) - \varphi_k(\hat{z}_k),x^{*} - \hat{z}_k \rangle \le 0.
\end{align}
By Assumption \ref{assumption-3}, we know that for any $k\in N$ and $d\in \mathbb{R}^n$, we have$\langle \nabla F(x_k)d,d\rangle \ge \mu \|d\|^2.$
Combining \eqref{4.6} and \eqref{super-2}, we obtain that
\begin{flalign}\label{super-3}
&\mu\|\hat{z}_k - x^{*}\|^2   \nonumber \\
\le &\langle \nabla F(x_k)(\hat{z}_k - x^{*}),\hat{z}_k - x^{*}\rangle \nonumber \\
\le &\langle R_k(x^{*} - x_k),\hat{z}_k-x^{*}\rangle + \langle(B_k - \nabla F(x_k)+\mu_k I)(\hat{z}_k - x_k),x^{*}-\hat{z}_k\rangle \nonumber\\
\le &\frac{L_2}{2}\|x^{*} - x_k\|^2\|\hat{z}_k-x^{*}\| +\|B_k - \nabla F(x_k)+\mu_k I\|\|\hat{z}_k-x_k\|\|\hat{z}_k-x^{*}\|.
\end{flalign}
Dividing both sides by $\|\hat{z}_k - x^{*}\|$, it gives that
\begin{align}\label{super-4}
&\mu\|\hat{z}_k-x^{*}\|  \nonumber\\
\le &\frac{L_2}{2}\|x^{*} - x_k\|^2 + \|B_k - \nabla F(x_k)+\mu_k I\|\|\hat{z}_k-x_k\| \nonumber \\
\le &\frac{L_2}{2}\|x^{*} - x_k\|^2 + \|B_k - \nabla F(x_k)\|\|\hat{z}_k-x_k\| + \mu_k \|\hat{z}_k-x_k\| \nonumber \\
\le &\frac{L_2}{2}\|x^{*} - x_k\|^2 + \mu_k \|\hat{z}_k-x_k\|  +  \|B_k - \nabla F(x_k)\| \|\hat{z}_k-z_k\| \nonumber\\
&+\|B_k - \nabla F(x_k)\|\|z_k-x_k\|.
\end{align}

\textbf{Step 2.}
Next, we analyze upper bounds for $\|\hat{z}_k-x_k\|,\ \|z_k-x_k\|,\ \|\hat{z}_k-z_k\|$ to establish their relationships with $\|x_k-x^{*}\|$.
From Lemma \ref{lemma-localineq} and assumption (ii) , there exists a constant $C_2>0$ such that
\begin{align}\label{super-5}
\|z_k - \hat{z}_k\| \le \frac{1+\|B_k + \mu_kI\|}{c_k} \|z_k - P_{C}(z_k - \varphi_k(z_k))\| \le C_2\|e_k\|,
\end{align}
where $c_k$ is the smallest eigenvalue of $B_k + \mu_kI$.
From assumption (iii) and $\|\nabla F(x_k) - B_k\| \le D$, together with \eqref{super-5}, \eqref{super-4}, and \eqref{algorithm-error1}, there exist $\{\delta_k\}$ such that
\begin{align}\label{super-4-2}
\|\hat{z}_k - x^{*}\| \le &\  \mu^{-1}\left(\frac{L_2}{2}\|x^{*} - x_k\|^2 + \mu_k \|\hat{z}_k-x_k\| \right) \nonumber\\
&+ \mu^{-1}(C_2D\rho_k\mu_k + \delta_k)\|z_k - x_k\|.
\end{align}
With the triangle inequality  $\|\hat{z}_k - x_k\| \le \|\hat{z}_k - x^{*}\| + \|x_k - x^{*}\|$
and \eqref{super-4-2}, we obtain
\begin{align*}
  \|\hat{z}_k - x_k\| \le &\ \frac{L_2}{2\mu}\|x^{*} - x_k\|^2 + \frac{\mu_k}{\mu} \|\hat{z}_k-x_k\| \nonumber \\
   &+ \mu^{-1}(C_2D\rho_k\mu_k + \delta_k)\|z_k - x_k\| + \|x_k - x^{*}\|  .
\end{align*}

Therefore, we have
\begin{align}\label{super-6}
\left(1-\frac{\mu_k}{\mu}\right)\|\hat{z}_k-x_k\| \le &\frac{L_2}{2\mu}\|x_k-x^{*}\|^2+\|x_k-x^{*}\| \nonumber\\
&+\frac{1}{\mu}(C_2D\rho_k\mu_k+\delta_k)\|z_k-x_k\|.
\end{align}

Consider \eqref{super-6} and assumption (iii), for sufficiently large $k$, we have
\begin{align}\label{super-7}
\|\hat{z}_k-x_k\|\le &\frac{L_2}{\mu}\|x^{*}-x_k\|^2+2\|x_k-x^{*}\| +\frac{2}{\mu}\left(C_2D\rho_k\mu_k+\delta_k\right)\|z_k-x_k\|.
\end{align}

From the triangle inequality and \eqref{super-5}, together with \eqref{algorithm-error1} in the algorithm, it holds that
$\|z_k-x_k\| \le\|\hat{z}_k-z_k\|+\|\hat{z}_k-x_k\| \le C_2\rho_k\mu_k\|z_k-x_k\| + \|\hat{z}_k-x_k\|.$
Since $\rho_k\mu_k \to 0$, for $k $ sufficiently large, it holds that $\max{(C_2\rho_k\mu_k, C_2D\rho_k\mu_k)} \le \min{(\frac{1}{2} , \frac{\mu}{16})}$.
The above inequality directly gives that
\begin{align}\label{eq1}
\|z_k-x_k\|\le 2\|\hat{z}_k-x_k\|.
\end{align}
Combining \eqref{eq1} and \eqref{super-7} yields
\begin{align}\label{super-8}
\|z_k-x_k\|\le &\frac{2L_2}{\mu}\|x_k-x^{*}\|^2+4\|x_k-x^{*}\| +\frac{4}{\mu}\left(C_2D\rho_k\mu_k + \delta_k\right)\|z_k-x_k\|.
\end{align}
It leads to the following result
$$\left(1-\frac{4}{\mu}(C_2D\rho_k\mu_k+\delta_k) \right)\|z_k-x_k\|\le 4\left( \frac{L_2}{2\mu}\|x_k-x^{*}\|^2+\|x_k-x^{*}\| \right).$$
Since $\delta_k \le \frac{\mu}{16}$, we have $\delta_k + C_2D\rho_k\mu_k \le \frac{1}{8}\mu$. Thus, we obtain
\begin{align*}
\|z_k - x_k\| \le 8\left( \frac{L_2}{2\mu}\|x_k-x^{*}\|^2+\|x_k-x^{*}\| \right).
\end{align*}
As $\|x_k - x^{*}\|\to 0$, for $k$ sufficiently large, it holds that
\begin{align}\label{super-9}
\|z_k-x_k\|\le 16\|x_k-x^{*}\|.
\end{align}
Substituting \eqref{super-9} into \eqref{super-7}, we obtain that
\begin{align*}
\|\hat{z}_k-x_k\| \le 2\left( \frac{L_2}{2\mu}\|x^{*}-x_k\|^2+\|x_k-x^{*}\|
+2\|x_k-x^{*} \| \right).
\end{align*}
Then, fot $k$ sufficiently large
\begin{align}\label{super-10}
\|\hat{z}_k - x_k\| \le  7\|x_k - x^{*}\|.
\end{align}
Therefore, we obtained upper bounds for $\|\hat{z}_k-x_k\|, \  \|z_k-x_k\|$, and $\|\hat{z}_k-z_k\|$. Consider \eqref{super-4-2}, \eqref{super-9}, and \eqref{super-10}, it holds that
\begin{align}\label{super-11}
\|\hat{z}_k-x^{*}\|\le & \frac{L_2}{2\mu}\|x_k-x^{*}\|^2
+\frac{7\mu_k}{\mu}\|x_k-x^{*}\| +\frac{1}{\mu}16(C_2D\rho_k\mu_k+\delta_k)\|x_k-x^{*}\|  .
\end{align}

\textbf{Step 3.} We will derive an upper bound for $\|z_k - z^*\|$.
By assumption (i), we have $\mu_k \le C_1 \|x_k - x^*\|,\ \rho_k \to 0$. For $k$ sufficiently large, it holds that $\rho_k < \frac{1}{16C_2}$. Combining \eqref{super-5} and \eqref{super-11} and using the triangle inequality, it holds that
\begin{align}\label{super-12}
\|z_k-x^{*}\| &\le \|z_k - \hat{z}_k\| + \|\hat{z}_k - x^{*}\| \nonumber \\
&\le  \frac{L_2}{2\mu}\|x_k-x^{*}\|^2
+\frac{5\mu_k}{\mu}\|x_k-x^{*}\| +\frac{16}{\mu}(C_2D\rho_k\mu_k+\delta_k)\|x_k-x^{*}\|  \nonumber \\
&\quad+ 16C_2\rho_k\mu_k\|x_k - x^{*}\| \nonumber\\
&\le \frac{1}{\mu}\left( 16C_2\rho_k\mu_k + 5\mu_k + 16(C_2D\rho_k\mu_k+\delta_k) \right)\|x_k - x^{*}\| \nonumber\\
& \quad +\frac{L_2}{2\mu}\|x_k - x^{*}\|^2.
\end{align}

Similarly, for $k$ sufficiently large
\begin{align*}
\|z_k - x^{*}\| \le \frac{12C_1+ L_2 + 2D}{2\mu}\|x_k - x^{*}\|^2 + 16\delta_k\|x_k - x^{*}\|,
\end{align*}
which is \eqref{4.4}. This completes the proof.
\end{proof}

\begin{lemma}\label{lemma-localunit0}\rm{ \cite{Taji2008} }
Suppose that Assumption \ref{assumption-1} is satisfied. Then there exists a constant $\bar{L} > 0$ such that for any $x\in X$, we have
\begin{align}
f_{\alpha}(x) \le \langle F(x^{*}),x-x^{*} \rangle + \bar{L}\|x-x^{*}\|^2.
\end{align}
\end{lemma}

\begin{theorem}\label{lemma-localunit1}
Suppose all the assumptions in Theorem \ref{lemma-localconverge} are satisfied in a local neighborhood of the solution $x^{*}$ of \eqref{VIP} . Let $\{z_k\}$ be the approximate solution to the subproblem \ref{subproblem} in Algorithm \ref{algorithm}, and let $\{\hat{z}_k\}$ be the exact solution. Then, in a local neighborhood $X$ of $x^{*}$, there exist constants $L_1',\ L_2',\ L_3',\ L_4',\ L_5' > 0$ such that
\begin{align}\label{localunit1-1}
f_{\alpha}(z_k) \le & L_1'\|x_k - x^{*}\|^4 + (L_2'\delta_k + L_3')\|x_k - x^{*}\|^3 + (L_4'\delta_k + L_5'\delta_k^2)\|x_k - x^{*}\|^2.
\end{align}
\end{theorem}
\begin{proof}
According to Theorem \ref{lemma-global}, $\{x_k\}$ is bounded. Suppose there exist constants $D_0, \  D_1>0$ such that for any $x_k \in X, \ \|B_k\| \le D_0, \  \|F(x_k)\|\le D_1$. Let $\bar{L}= \frac{1}{2}L$. Since $z_k$ is an approximate solution of \ref{subproblem}, it holds that
\begin{align}
\langle F(x_k),z_k - x^{*}\rangle &= \langle F(x_k),z_k - \hat{z}_k + \hat{z}_k - x^{*}\rangle \nonumber \\
&\le \|F(x_k)\| \cdot \|z_k - \hat{z}_k\| + \langle F(x_k),\hat{z}_k - x^{*}\rangle.
\end{align}
According to Theorem \ref{lemma-localconverge}, there exist constants $\gamma_1,\gamma_2 >0$ such that
\begin{align}\label{3.47}
\|\hat{z}_k - x^{*}\| \le \gamma_1 \|x_k - x^{*}\|^2 + \gamma_2\delta_k\|x_k - x^{*}\|.
\end{align}
On the other hand, based on \eqref{algorithm-error1} and Lemma \ref{lemma-localineq}, there exists a constant $\mu$ such that the following inequality holds $\|z_k - \hat{z}_k\| \le \mu  \rho_k \mu_k \|z_k - x_k\|.$
According to the definitions of $\rho_k$ and $\mu_k$ in Theorem \ref{lemma-localconverge} and Lemma \ref{lemma-project}, there exists a constant $C_2 > 0$ such that
\begin{align}\label{3.48}
\|z_k - \hat{z}_k\| \le C_2 \|x_k - x^{*}\|^3.
\end{align}
According to the proof of Theorem \ref{lemma-localconverge}, for $k$ sufficiently large, the following holds
\begin{align}\label{eq-1}
\|\hat{z}_k - x_k\| \le 7\|x_k - x^{*}\|.
\end{align}
Since $\hat{z}_k$ is the exact solution to $VIP(\varphi_k,C)$, combining \eqref{3.47},\  \eqref{3.48}, and Lemma \ref{lemma-localunit0}, we obtain that
\begin{align*}
f_{\alpha}(z_k) &\le \langle F(x^{*}),z_k - x^{*} \rangle + \bar{L}\|z_k - x^{*}\|^2 \\
&= \langle F(x^{*}) - F(x_k), z_k - x^{*}\rangle + \langle F(x_k),z_k - x^{*}\rangle + \bar{L}\|z_k -x^{*}\|^2 \\
&\le \langle F(x^{*}) - F(x_k), z_k - x^{*}\rangle + \|F(x_k)\|\|z_k - \hat{z}_k\| + \bar{L}\|z_k-x^{*}\|^2  \\
&\quad+ \langle F(x_k), \hat{z}_k - x^{*}\rangle \\
&\le \|F(x^{*})-F(x_k)\|\|z_k - x^{*}\| + \langle (B_k + \mu_kI)(\hat{z}_k - x_k),x^{*} - \hat{z}_k\rangle \\
&\quad   + \|F(x_k)\|\|z_k - \hat{z}_k\| + \bar{L}\|z_k - x^{*}\|^2 \\
&\le L_1\gamma_1\|x_k - x^{*}\|^3 + L_1\gamma_2\delta_k\|x_k-x^{*}\|^2 +\bar{L}\|z_k - x^{*}\|^2  \\
&\quad  + D_1C_2\|x_k - x^{*}\|^3 + 7(D_0 + \mu_k)\|x_k - x^{*}\|\|\hat{z}_k - x^{*}\| .
\end{align*}

According to Theorem \ref{lemma-localconverge}, there exist $\gamma_3 > 0,\ \gamma_4 > 0$ such that
\begin{align}\label{3.49}
\|z_k - x^{*}\| \le \gamma_3\|x^{*}-x_k\|^2 + \gamma_4\delta_k\|x_k-x^{*}\|.
\end{align}
Thus, we have
\begin{align*}
\bar{L}\|z_k - x^{*}\|^2 \le \bar{L} \left( \gamma_3^2\|x_k - x^{*}\|^4 + 2\gamma_3\gamma_4\delta_k\|x_k - x^{*}\|^3 + \gamma_4^2\delta_k^2\|x_k-x^{*}\|^2 \right).
\end{align*}
Together with \eqref{3.47} and noting that $\mu_k \to 0$, we can easily obtain the following result
\begin{align*}
7(D_0+\mu_k)\|x_k-x^{*}\|\|\hat{z}_k-x^{*}\|\le 8D_0(\gamma_1\|x_k-x^{*}\|^3+\gamma_2\delta_k\|x_k-x^{*}\|^2).
\end{align*}
After a simple calculation, we obtain that there exist $L_ {1}',L_{2}',K_{3},L_{4}',L_{5}'\ge0 $ such that \eqref{localunit1-1} holds
where, $L_{1}'=\bar{L}\gamma_3^2,\ L_{2}'=2\bar{L}\gamma_3\gamma_4,\ L_{3}'=8D_0\gamma_1 + L_1\gamma_1 + D_1C_2,\  L_{4}'=8D_0\gamma_2 + L_1\gamma_2 $,\  $L_{5}'=\bar{L}\gamma_4^2$. The proof is finished.
\end{proof}

The following result states that by utilizing the error bound of the merit function $f_{\alpha}(z_k)$, the algorithm will automatically choose the unit step size within some local neighborhood of the solution of VIP\eqref{VIP}.

\begin{theorem}\label{lemma-localunit2}
Assume that all the conditions in Theorem \ref{lemma-localconverge} hold within some local neighborhood $X$ of the solution $x^*$ of VIP\eqref{VIP}. Let $0< \gamma <1$. If the conditions in Proposition \ref{p5} hold, then there exists a constant $0 < \delta < \frac{\mu}{16}$ and a neighborhood $X'$ of $x^{*}$, such that for any $x_k \in X'$ and $\delta_k \le \delta$, we have $f_{\alpha}(z_k)\le \gamma f_{\alpha}(x_k)$.
\end{theorem}
\begin{proof}
By Theorem \ref{lemma-localconverge}, there exist constants $\gamma_1,\ \gamma_2,\ \gamma_3,\ \gamma_4>0$ such that within some local neighborhood $N$ of $x^{*}$, it holds that
\begin{align*}
\begin{cases}
\|\hat{z}_k - x^{*}\| \le \gamma_1 \|x^{*} - x_k\|^2+\gamma_2 \delta_k \|x_k - x^{*}\|, \\
\|z_k - x^{*}\|\le \gamma_3 \|x_k - x^{*}\|^2 + \gamma_4 \delta_k \|x_k - x^{*}\|.
\end{cases}
\end{align*}

By Theorem \ref{lemma-localunit1}, we know that there exist $L_ {1}',L_{2}',K_{3},L_{4}',L_{5}'\ge0 $ such that within some local neighborhood $N'$ of $x^{*}$, \eqref{localunit1-1} holds.
According to Proposition \ref{p5}, for any $x_k \in X$, it holds that
$$f(x_k) \ge (\mu - \frac{\alpha}{2})\|x_k - x^{*}\|^2.$$
Combining with \eqref{localunit1-1}, for any $x\in (N \cap N')$, it holds that
$$f_{\alpha}(z_k) \le \left( L_{1}'\|x_k - x^{*}\|^4 + (L_{2}'\delta_k + L_{3}')\|x_k - x^{*}\|^3 + (L_{4}'\delta_k + L_{5}'\delta_{k}^2) \right) \frac{2f(x_k)}{2\mu - \alpha}.$$
Since $\|x_k - x^{*}\|\to 0$, for $k$ sufficiently large, it holds that
$$\frac{L_{1}'\|x_k - x^{*}\|^4 + (L_{2}'\delta_k + L_{3}')\|x_k - x^{*}\|^3}{\mu - \alpha/2} \le \frac{\gamma}{2}.$$
If $0 < \delta_k \le \sqrt{\frac{\gamma(\mu - \alpha /2)}{2L_5'}+\frac{L_4'^2}{4L_5'^2}} - \frac{L_4'}{2L_5'}$, then $\frac{L_4'\delta_k + L_5'\delta_k^2}{\mu - \alpha/2} \le \frac{\gamma}{2}$, and thus
$$f_ {\alpha}(z_k) \le \gamma f(x_k).$$
Therefore, there exists a local neighborhood $X'=(N \cap N')$ of the solution to problem \eqref{VIP}, such that for any $x_k \in X'$, we have
$$\frac{L_{1}'\|x_k - x^{*}\|^4 + (L_{2}'\delta_k + L_{3}')\|x_k - x^{*}\|^3 + (L_{4}'\delta_k + L_{5}'\delta_{k}^2)}{\mu-\alpha /2} \le \gamma.$$
Let $\delta = \min{\left( \sqrt{\frac{\gamma(\mu - \alpha /2)}{2L_5'}+\frac{L_4'^2}{4L_5'^2}} - \frac{L_4'}{2L_5'} , \frac{\mu}{16} \right)}$. Therefore, when $\delta_k \le \delta$, for any $x_k \in X'$, we have $f_{\alpha}(z_k) \le \gamma f_{\alpha}(x_k)$. The proof is finished.
\end{proof}

\section{Numerical Experiment}
\numberwithin{equation}{section}
In this section, we report the numerical results of the inexact regularized quasi-Newton method (IRQN) proposed in this paper and compare it with the inexact Newton method (INM) proposed by \cite{solodov2000truly}. The code of the algorithm is implemented in MATLAB R2020a and computed on a personal computer with an AMD Ryzen 7 5800H with Radeon Graphics 3.20 GHz CPU.

The parameter settings of the algorithm are as follows. For IRQN, we set $\rho_k = 0$, $\lambda = 0.5$, $\eta = 0.3$, $\beta = 0.7$, $\gamma = 0.5$, $h = 10^{-5}$, $\alpha = 0.01$ and $B_0 = I$, where $I\in R^{n\times n}$ is the identity matrix. In Algorithm \ref{algorithm}, we use the BFGS method in \cite{2000LiDH} to update $B_k$:
\begin{align*}%\label{broyden}
B_{k+1} =
\begin{cases}
B_{k}-\frac{B_{k}s_{k}s_{k}^{T}B_{k}}{s_{k}^{T}B_{k}s_{k}}+\frac{y_{k}y_{k}^{T}}{y_{k}^{T}s_{k}},  & \mbox{if }\frac{y_{k}y_{k}^{T}}{s_{k}^{T}s_{k}}\ge h\mu_k^{r}\mbox{;  } \\
B_k, & \mbox{otherwise.  }
\end{cases}
\end{align*}
Here, $s_k = x_{k+1} - x_k$, $y_k = F(x_{k+1}) - F(x_k)$, and $r,h > 0$ are constants.

In INM, we set $\rho_k = 0$, $\lambda = 0.5$, $\eta = 0.3$, and $\beta = 0.7$. In both algorithms, the same projection method \cite{HanBBVI} is used to solve the subproblems. When the constraint set $C$ is $\mathbb{R}_{+}^n$ or $[a,b]$, $P_{C}(x)$ has an explicit expression. If $C$ is a linear constraint, we will use the MATLAB quadratic programming solver 'quadprog' to calculate the projection $P_{C}(x)$. The regularization term of the subproblem is set to $\mu_k=0.01\times \|x_k - P_{C}(x_k - F(x_k)/0.01)\|$ for both algorithms. The stopping criterion for both algorithms is $\alpha\|x_k - P_{C}(x_k - F(x_k)/\alpha)\| < 10^{-5}$.

We will report the following results: dimension $n$, number of iterations $iter$, computation time $t$ (in seconds), and error $res$. Here, $res = \alpha\|x_k - P_{C}(x_k - F(x_k)/\alpha)\|$.

\subsection{Nolinear equations}
We test the following examples. We use $A = tridiag(a, b, c)$ to denote a tridiagonal matrix, where $A_{ii}= b, A_{ii+1}= c, A_{i, i-1} = a$.

\textbf{Ex 1} \cite{Zhou2008} $F$ is defined as: $F_{i}(x)=x_i - \sin{(x_i)},\  i=1,2,...,n,$ $x_0 = (1,1,\cdots,1)^{\top}$.

\textbf{Ex 2} \cite{Zhou2008} Here $F(x) = Ax + g(x)$, where $g(x) = (e^{x_1}-1,e^{x_2}-1,\dots,e^{x_n}-1)^{\top}$,  $A = tridiag(-1,2,-1) \in \mathbb{R}^{n\times n}$ and $x_0 = (1,1,\cdots,1)^{\top}$.

\begin{table*}[htbp]
\scriptsize
\centering
\caption{\text{Results on nonlinear equation problems}}\label{Tab01}
%\begin{tabular}{ccccccc}
\begin{tabular*}{\textwidth}{@{\extracolsep\fill}cccccccc}
\toprule
 %\multicolumn{7}{c}{\text{Example 1}}  \\
\multirow{2}{*}{} &\  & \multicolumn{3}{c}{INM} & \multicolumn{3}{c}{IRQN} \\
\cmidrule(r){3-5} \cmidrule(r){6-8}
&{n} &  iter     &  time   &   res

&  iter      &  time   &   res \\
\midrule
\multirow{3}{*}{\textbf{Ex 1} } &$100$ & 17 & 1.3e-3 & 1.7e-6 & 32 & 2.4e-3 & 9.4e-6   \\

\multirow{3}{*}{} &$1000$ & 40 & 1.1e-1 & 1.2e-7 & 55 & 6.8e-2 & 3.4e-6   \\

\multirow{3}{*}{}  &$5000$ & 80 & 7.4e+0 & 2.8e-8 & 94 & 6.6e+0 & 5.5e-5
\\
\midrule
\multirow{3}{*}{\textbf{Ex 2} }&$100$ & 18 & 5.3e-3 & 1.1e-6 & 29 & 3.8e-3 & 7.6e-6  \\

\multirow{3}{*}{}  &$1000$ & 40 & 6.0e-1 & 6.3e-6 & 51 & 4.1e-1 & 8.5e-6   \\

\multirow{3}{*}{}  &$5000$ & 64 & 7.3e+0 & 3.4e-7 & 71 & 4.8e+0 & 5.6e-6 \\
\bottomrule
\end{tabular*}
\end{table*}
The numerical results are reported in Table \ref{Tab01}, where the two methods produce similar results, and our proposed algorithm has relatively less computational time for dimensions of 1000 and 5000.

\subsection{Variational Inequality and Complementarity problems}
We test the following examples.

\textbf{Ex 3} \cite{Kebaili2018A} $x_0=(1,1,\cdots,1)^{\top}$, $F(x)=Mx + q ,$
where
$$
M=\begin{bmatrix}
1      &  2     &  2      & \cdots  & 2 \\
0      &  1     &  2      & \cdots  & 2 \\
0      &  0     &  1      & \cdots  & 2 \\
\vdots & \vdots &  \vdots & \ddots  & \vdots \\
0      & 0      &  0      & \cdots  & 1
\end{bmatrix}  \in \mathbb{R}^{n\times n}, \  q=\begin{bmatrix}
-1      \\
-1      \\
\vdots  \\
-1
\end{bmatrix} \in \mathbb{R}^n .
$$
It has a degenerate solution $x^*=(0,0,\dots,0,1)^{\top}$.

\textbf{Ex 4} \cite{HanBBVI} Initial point $x_0=(1,1,\dots,1)^{\top}$, $F(x) = \rho D(x) + Mx+q$, where $D(x)$ and $Mx+q$ are the nonlinear and linear part of $F(x)$, respectively. $M=A^TA+B$ with a random antisymmetric matrix $A\in \mathbb{R}^{n\times n}$ generated uniformly on $(-5,5)$, and a random matrix $B$ generated similarly. Vector $q$ is uniformly generated on $(-500,500)$. The components of $D(x)$ are $D_i(x)=a_i \arctan{(x_i)}$, where $a_i$ is a random number uniformly generated on $(0,1)$.

\textbf{Ex 5} \cite{Kebaili2018A} Constraint set $C= [0,1]^n$, initial point $x_0 = (-1,-1,\cdots,-1)^{\top}$. $F(x)=Mx + q$, where $M = tridiag(-1, 4, -1) \in \mathbb{R}^{n\times n}$, $q=(-1,-1,\cdots ,-1)^{\top} \in \mathbb{R}^{n}.$

\textbf{Ex 6} \cite{Harker1990A} Feasible set $C:=\{x\in R^m | \  Qx \le b\}$, with initial point $(1,1,\cdots,1)^{\top}$. $F(x)=Mx$, where $Q \in \mathbb{R}^{l\times m}$ is a random matrix and $b\in \mathbb{R}^l$ is a non-negative random vector. $M=ZZ^{\top}+S+D$, where $Z,S,D\in \mathbb{R}^{m \times m}$ are randomly generated matrices, $S$ is skew-symmetric, and $D$ is a positive definite diagonal matrix, such that the variational inequality problem has a unique solution.

\begin{table*}[htbp]
\scriptsize
\centering
\caption{\text{Results on variational inequalities and complementarity problems}}\label{Tab02}
\begin{tabular}{cccccccc}
\toprule
\multirow{2}{*}{} &\ & \multicolumn{3}{c}{INM} & \multicolumn{3}{c}{IRQN} \\
\cmidrule(r){3-5} \cmidrule(r){6-8}
&{n} &  iter     &  times   &   res

&  iter      &  times   &   res \\
\midrule
\multirow{3}{*}{\textbf{Ex 3} } &
$100$ & 2 & 4.7e-3 & 1.8e-15 & 2 & 1.5e-3 & 0   \\

\multirow{3}{*}{ } &$1000$ & 2 & 1.7e-2 & 0 & 2 & 2.3e-2 & 0    \\

\multirow{3}{*}{ } &$5000$ & 2 & 5.6e-1 & 3.3e-13 & 2 & 6.7e-1 & 0      \\
\midrule
\multirow{2}{*}{\textbf{Ex 4} } &
$100$ & 2001 & 5.6e-1 & 1.2e+0 & 208 & 6.3e-2 & 3.8e-6  \\

\multirow{2}{*}{ } &$1000$ & 2001 & 2.8e+1 & 1.3e+0 & 1771 & 4.5e+1 & 8.7e-6  \\

\midrule
\multirow{3}{*}{\textbf{Ex 5} } &
$100$ & 2 & 4.8e-4 & 4.4e-17 & 4 & 1.2e-3 & 0   \\

\multirow{3}{*}{ } &$1000$ & 2 & 2.5e-2 & 0 & 4 & 1.0e-1 & 0    \\

\multirow{3}{*}{ } &$5000$ & 2 & 7.3e-1 & 1.7e-15 & 4 & 3.1e+0 & 0      \\
\midrule
\multirow{2}{*}{\textbf{Ex 6} } &
$5$ & 53 & 1.0e-1 & 6.6e-6 & 38 & 9.3e-2 & 3.0e-6    \\

\multirow{2}{*}{ } &$10$ & 33 & 1.6e-1 & 9.7e-6 & 41 & 4.6e-2 & 6.1e-6   \\
\bottomrule
\end{tabular}
\end{table*}
The results for \textbf{EX 3} to \textbf{EX 6} are reported in Table \ref{Tab02}. From Table \ref{Tab02}, it shows that our proposed algorithm has advantages in terms of computational time and number of iterations.  It can be noticed that IRQN takes less number of iterations than INM, verifying the efficiency of merit function. Particular, for \textbf{Ex 4}, IRQN only takes about 200 iterations, which is much more smaller than that of INM.

\subsection{Impact of Initial Points}
We test the role of initial points on the following examples.

\textbf{Ex 7} \cite{Kebaili2018A} (Kojima-Shindo problem) The initial points are as follows  $(5,-1,1,1)^{\top}$, $(-1,-5,0,-3)^{\top}$, $(0.6,4,0,8)^{\top}$, $(1,-2,0.7,1)^{\top}$, $(1,-6,5,3)^{\top}$, and $(-1,-1,-1,-1)^{\top}$. The feasible set is $C=[-0.5, 0.5]^{4}$. $F$ is defined as:
\begin{align*}
F(x_1,x_2,x_3,x_4)=\begin{pmatrix}
3x_1^2 + 2x_1x_2 + 2x_2^2 + x_3 +3x_4 - 6 \\
2x_1^2 + x_1 + x_2^2 + 10x_3 + 2x_4 - 2 \\
3x_1^2+x_1x_2+2x_2^2+2x_3+9x_4-9 \\
x_1^2 + 3x_2^2 + 2x_3 +3x_4 - 3
\end{pmatrix}.
\end{align*}

\textbf{Ex 8} \cite{Kebaili2018A} Constraint set $C= [0,5]^4$, the initial points $(1,1,1,1)^{\top}$, $(-1,-1,-1,-1)^{\top}$, $(-6,-6,-10,-1)^{\top}$. $F$ is defined as: $$
F(x)=\begin{pmatrix}
x_1^3 - 8 \\
x_2 - x_3 + x_2^3 + 3 \\
x_2 - x_3 + 2x_3^3 - 3 \\
x_4 - 2x_4^3
\end{pmatrix}.
$$ This problem has degenerate solutions $x^{*}=(2,0,1,0)^{\top}$ and $x^{*}=(2,0,1,5)^{\top}$.

\textbf{Ex 9} \cite{Kebaili2018A} Constraint set $C= [-10,10]^4$. Initial point: $(3,3,3,3)^{\top}$, $(0,0,0,0)^{\top}$, $(1,1,1,1)^{\top}$. $F$ is defined as: $$
F(x)=\begin{pmatrix}
400x_1^3 + 2x_1 - 400x_1x_2 - 2 \\
-200x_1^2 + 200.2x_2 + 19.8x_4 - 40 \\
360x_1^3 + 2x_2 - 360x_3x_4 - 2 \\
19.8x_2 - 180x_3^2 + 220.2x_4^2 - 40
\end{pmatrix}.
$$

\textbf{Ex 10} \cite{majig2008restricted}  Constraint set is given by $$ C=
\left \{
x\in \mathbb{R}^n \  | \  \sum_{i=1}^n x_i \le n,\  \sum_{i=1}^n ix_i \ge n+1, x_i \ge 0, i=1,2,\cdots , n-1
\right \}.$$
The initial points are $(10,0,0,0,0)^{\top}$, $(10,0,10,0,10)^{\top}$, $(25,0,0,0,0)^{\top}$. $F(x) = Mx+q$, where
$$
M=\begin{bmatrix}
1 & 0 & 0 & 0 & 1\\
0 & 1 & 0 & 0 & 1\\
0 & 0 & 1 & 0 & 0\\
0 & 0 & 1 & 1 & 0
\end{bmatrix} , \  q=\begin{bmatrix}
-1\\
-1\\
-0.5\\
-0.5\\
-1
\end{bmatrix} .
$$

\textbf{Ex 11} \cite{TajiA1993} Constraint set is given by
$$C=\begin{Bmatrix}
x\in R^5 \mid 10 \le \sum_{k=1}^5 x_i \le 50,x_i\ge 0,i=1,2,3,4,5
\end{Bmatrix}.$$

The initial points are  $(25,0,0,0,0)^{\top}$, $(10,0,10,0,10)^{\top}$, $(0,2.5,2.5,2.5,2.5)^{\top}$, $(10,0,0,0,0)^{\top}$. $F(x)=Mx+\rho *C(x)+q ,$ where $C(x) = arctan (x-2)$,
\begin{align*}
M&=
\begin{bmatrix}
 0.726& -0.949& 0.266& -1.193& -0.504\\
 1.645& 0.678& 0.333& -0.217& -1.443\\
 -1.016& -0.225& 0.769& 0.934& 1.007\\
 1.063& 0.567& -1.144& 0.550& -0.548\\
 -0.259& 1.453& -1.073& 0.509& 1.026
\end{bmatrix},
\ q=
\begin{bmatrix}
5.308\\
0.008\\
-0.938\\
1.024\\
-1.312
\end{bmatrix}.
\end{align*}
This problem has a unique solution $x^{*}=(2,2,2,2,2)^{\top}$.

\textbf{Ex 12} \cite{TajiA1993} Constraint set $$C=\{x\in \mathbb{R}^5 \mid Ax\le b,x_i\ge 0,i=1,2,3,4,5\},$$ where
\begin{align*}
A=\begin{bmatrix}
&0   &0  &-0.5 &0    &-2\\
&-2  &-2 &0    &-0.5 &-2\\
&2   &2  &-4   &2    &-3\\
&-5  &3  &-2   &0    &2
\end{bmatrix}, \
b=\begin{bmatrix}
-10\\ -10\\ 13\\ 18
\end{bmatrix}.
\end{align*}
Initial points are $(0,0,100,0,0)^{\top}$, $(10,0,10,0,10)^{\top}$, $(0,2.5,2.5,2.5,2.5)^{\top}$. $F(x)=Mx + D(x) + q $, where
\begin{align*}
M=\begin{bmatrix}
    &3   &-4  &-16  &-15  &-4 \\
    &4   &1   &-5   &-10  &-11 \\
    &16  &5   &2    &-11  &-7 \\
    &15  &10  &11   &3    &-10 \\
    &4   &11  &7    &10   &1
    \end{bmatrix},\
D(x)=\begin{bmatrix}
    0.004x(1)^4 \\
    0.007x(2)^4 \\
    0.005x(3)^4 \\
    0.009x(4)^4 \\
    0.008x(5)^4 \\
    \end{bmatrix}, \
q=\begin{bmatrix}
-15\\ 10\\ -50\\ -30\\ -25
\end{bmatrix}.
\end{align*}
This problem has a unique solution $x^{*}=(9.08,4.84,0.00,0.00,5.00)^{\top}$.

\begin{table*}[htbp]
\scriptsize
\centering
\caption{Variational inequalities and complementarity problems}\label{Tab03}
%\begin{tabular}{ccccccc}
\begin{tabular*}{\textwidth}{@{\extracolsep\fill}cccccccc}
\toprule
\multirow{2}{*}{} &\ & \multicolumn{3}{c}{INM} & \multicolumn{3}{c}{IRQN} \\
\cmidrule(r){3-5} \cmidrule(r){6-8}
&{$x_0$} &  iter     &  time   &   res

&  iter      &  time   &   res \\
\midrule
\multirow{6}{*}{\textbf{Ex 7} } &
$(5,-1,1,1)$ & 4 & 2.0e-4 & 4.6e-6 & 11 & 6.2e-5 & 3.0e-7   \\

\multirow{6}{*}{} &$(-1,-5,0,-3)$ & 20 & 2.8e-3 & 4.8e-7 & 6 & 1.7e-3 & 5.6e-10    \\

\multirow{6}{*}{} &$(0.6,4,0,8)$ & 101 & 1.2e+0 & 3.4e-3 & 11 & 1.7e-3 & 1.4e-7      \\

\multirow{6}{*}{} &$(1,-2,0.7,1)$ & 4 & 9.6e-3 & 9.1e-7 & 5 & 4.8e-4 & 9.1e-7      \\

\multirow{6}{*}{} &$(1,-6,5,3)$ & 4 & 3.2e-4 & 9.1e-7 & 7 & 4.0e-4 & 1.0e-6      \\

\multirow{6}{*}{} &$(-1,-1,-1,-1)$ & 17 & 1.0e-3 & 4.7e-7 & 12 & 9.8e-4 & 2.0e-10\\
\midrule
\multirow{3}{*}{\textbf{Ex 8} } &
$(-1,-1,-1,-1)$& 101 & 1.1e-2 & 4.3e-2 & 12 & 9.8e-4 & 3.8e-6   \\

\multirow{3}{*}{} &$(1,1,1,1)$ & 101 & 9.3e-3 & 4.6e-2 & 12 & 9.3e-4 & 7.1e-6    \\

\multirow{3}{*}{} &$(-6,-6,-10,-1)$ & 101 & 1.1e-2 & 4.3e-2 & 18 & 1.3e-3 & 1.3e-7  \\
\midrule
\multirow{3}{*}{\textbf{Ex 9} } &
$(0,0,0,0)$ & 101 & 7.3e-1 & 1.9e-1 & 43 & 4.2e-2 & 9.3e-6   \\

\multirow{3}{*}{} &$(1,1,1,1)$ & 101 & 1.3e+0 & 2.3e-1 & 101 & 4.0e-1 & 1.6e-1 \\

\multirow{3}{*}{} &$(3,3,3,3)$ & 101 & 9.4e-1 & 2.3e-1 & 101 & 4.0e-1 & 1.6e-1 \\
\midrule
\multirow{3}{*}{\textbf{Ex 10} } &
$(10,0,0,0,0)$ & 18 & 9.9e-3 & 6.6e-6 & 21 & 6.0e-3 & 1.1e-6   \\

\multirow{3}{*}{} &$(10,0,10,0,10)$ & 30 & 1.2e-2 & 8.5e-6 & 39 & 9.2e-3 & 1.0e-6 \\

\multirow{3}{*}{} &$(25,0,0,0,0)$ & 36 & 1.0e-2 & 3.4e-6 & 36 & 8.6e-3 & 5.6e-7 \\
\midrule
\multirow{4}{*}{\textbf{Ex 11} } &
$(25,0,0,0,0)$ & 101 & 3.5e-2 & 4.8e-1 & 17 & 2.2e-2 & 3.3e-6   \\

\multirow{4}{*}{} &$(10,0,10,0,10)$ & 101 & 3.3e-2 & 4.8e-1 & 22 & 1.7e-2 & 7.6e-6 \\

\multirow{4}{*}{} &$(0,2.5,2.5,2.5,2.5)$ & 101 & 3.4e-2 & 4.8e-1 & 15 & 1.5e-2 & 4.2e-6 \\

\multirow{4}{*}{} &$(10,0,0,0,0)$ & 101 & 3.5e-2 & 4.8e-1 & 20 & 1.5e-2 & 1.9e-6 \\
\midrule
\multirow{3}{*}{\textbf{Ex 12} } &
$(0,0,100,0,0)$ & 101 & 6.0e-2 & 1.1e-1 & 24 & 1.8e-2 & 1.1e-6   \\

\multirow{3}{*}{} &$(10,0,10,0,10)$ & 101 & 3.5e-2 & 1.1e-1 & 24 & 2.1e-2 & 1.4e-7 \\

\multirow{3}{*}{} &$(0,2.5,2.5,2.5,2.5)$ & 101 & 4.7e-2 & 7.4e-2 & 56 & 1.6e-2 & 2.1e-6 \\
\bottomrule
\end{tabular*}
\end{table*}

We test the impact of different initial points under the same dimension. Table \ref{Tab03} shows that for \textbf{Ex 8}, \textbf{Ex 11} and \textbf{Ex 12}, INM did not achieve the required solution accuracy whereas IRQN returns solutions within reasonable tolerance. Moreover, IRQN has relatively less computation time.

\section{Conclusion}\label{sec13}
In this paper, we proposed a globally convergent inexact regularized quasi-Newton method (IRQN) for monotone (non-strongly monotone) variational inequality problems. The main difference between IRQN and Solodov's method (INM) \cite{solodov2000truly} lies in the use of a quasi-Newton method, as well as the implementation of a merit function. By utilizing the error bounds of the merit function, we prove that the iterative sequence of the algorithm can eventually converge to the solution of the problem and the numerical results reveal the efficiency of the proposed method.

%%===========================================================================================%%
%% If you are submitting to one of the Nature Portfolio journals, using the eJP submission   %%
%% system, please include the references within the manuscript file itself. You may do this  %%
%% by copying the reference list from your .bbl file, paste it into the main manuscript .tex %%
%% file, and delete the associated \verb+\bibliography+ commands.                            %%
%%===========================================================================================%%

\bibliography{sn-bibliography}% common bib file

%% BioMed_Central_Bib_Style_v1.01

\begin{thebibliography}{25}
% BibTex style file: bmc-mathphys.bst (version 2.1), 2014-07-24
\ifx \bisbn   \undefined \def \bisbn  #1{ISBN #1}\fi
\ifx \binits  \undefined \def \binits#1{#1}\fi
\ifx \bauthor  \undefined \def \bauthor#1{#1}\fi
\ifx \batitle  \undefined \def \batitle#1{#1}\fi
\ifx \bjtitle  \undefined \def \bjtitle#1{#1}\fi
\ifx \bvolume  \undefined \def \bvolume#1{\textbf{#1}}\fi
\ifx \byear  \undefined \def \byear#1{#1}\fi
\ifx \bissue  \undefined \def \bissue#1{#1}\fi
\ifx \bfpage  \undefined \def \bfpage#1{#1}\fi
\ifx \blpage  \undefined \def \blpage #1{#1}\fi
\ifx \burl  \undefined \def \burl#1{\textsf{#1}}\fi
\ifx \doiurl  \undefined \def \doiurl#1{\url{https://doi.org/#1}}\fi
\ifx \betal  \undefined \def \betal{\textit{et al.}}\fi
\ifx \binstitute  \undefined \def \binstitute#1{#1}\fi
\ifx \binstitutionaled  \undefined \def \binstitutionaled#1{#1}\fi
\ifx \bctitle  \undefined \def \bctitle#1{#1}\fi
\ifx \beditor  \undefined \def \beditor#1{#1}\fi
\ifx \bpublisher  \undefined \def \bpublisher#1{#1}\fi
\ifx \bbtitle  \undefined \def \bbtitle#1{#1}\fi
\ifx \bedition  \undefined \def \bedition#1{#1}\fi
\ifx \bseriesno  \undefined \def \bseriesno#1{#1}\fi
\ifx \blocation  \undefined \def \blocation#1{#1}\fi
\ifx \bsertitle  \undefined \def \bsertitle#1{#1}\fi
\ifx \bsnm \undefined \def \bsnm#1{#1}\fi
\ifx \bsuffix \undefined \def \bsuffix#1{#1}\fi
\ifx \bparticle \undefined \def \bparticle#1{#1}\fi
\ifx \barticle \undefined \def \barticle#1{#1}\fi
\bibcommenthead
\ifx \bconfdate \undefined \def \bconfdate #1{#1}\fi
\ifx \botherref \undefined \def \botherref #1{#1}\fi
\ifx \url \undefined \def \url#1{\textsf{#1}}\fi
\ifx \bchapter \undefined \def \bchapter#1{#1}\fi
\ifx \bbook \undefined \def \bbook#1{#1}\fi
\ifx \bcomment \undefined \def \bcomment#1{#1}\fi
\ifx \oauthor \undefined \def \oauthor#1{#1}\fi
\ifx \citeauthoryear \undefined \def \citeauthoryear#1{#1}\fi
\ifx \endbibitem  \undefined \def \endbibitem {}\fi
\ifx \bconflocation  \undefined \def \bconflocation#1{#1}\fi
\ifx \arxivurl  \undefined \def \arxivurl#1{\textsf{#1}}\fi
\csname PreBibitemsHook\endcsname

%%% 1
\bibitem[\protect\citeauthoryear{Harker}{1990}]{Pang1990}
\begin{barticle}
\bauthor{\bsnm{Harker}, \binits{J.-S.} \bsuffix{Patrick T.~Pang}}:
\batitle{Finite-dimensional variational inequality and nonlinear
  complementarity problems: A survey of theory, algorithms and applications}.
\bjtitle{Mathematical Programming}
\bvolume{48},
\bfpage{161}--\blpage{220}
(\byear{1990}).
\bcomment{doi:{\color{blue}
  \href{https://doi.org/10.1007/BF01582255}{10.1007/BF01582255}}}
\end{barticle}
\endbibitem

%%% 2
\bibitem[\protect\citeauthoryear{Marcotte and
  Dussault}{1987}]{marcotte1987note}
\begin{barticle}
\bauthor{\bsnm{Marcotte}, \binits{P.}},
\bauthor{\bsnm{Dussault}, \binits{J.-P.}}:
\batitle{A note on a globally convergent newton method for solving monotone
  variational inequalities}.
\bjtitle{Operations Research Letters}
\bvolume{6}(\bissue{1}),
\bfpage{35}--\blpage{42}
(\byear{1987})
\doiurl{10.1016/0167-6377(87)90007-1}
\end{barticle}
\endbibitem

%%% 3
\bibitem[\protect\citeauthoryear{Taji et~al.}{1993}]{TajiA1993}
\begin{barticle}
\bauthor{\bsnm{Taji}, \binits{K.}},
\bauthor{\bsnm{Fukushima}, \binits{M.}},
\bauthor{\bsnm{Ibaraki}, \binits{T.}}:
\batitle{A globally convergent newton method for solving strongly monotone
  variational inequalities}.
\bjtitle{Mathematical Programming}
\bvolume{58},
\bfpage{369}--\blpage{383}
(\byear{1993}).
\bcomment{doi:{\color{blue} \href{https://doi.org/10.1007/BF01581276}
  {10.1007/BF01581276}}}
\end{barticle}
\endbibitem

%%% 4
\bibitem[\protect\citeauthoryear{Peng and Fukushima}{1999}]{peng1999hybrid}
\begin{barticle}
\bauthor{\bsnm{Peng}, \binits{J.-M.}},
\bauthor{\bsnm{Fukushima}, \binits{M.}}:
\batitle{A hybrid newton method for solving the variational inequality problem
  via the d-gap function}.
\bjtitle{Mathematical Programming}
\bvolume{86}(\bissue{2}),
\bfpage{367}--\blpage{386}
(\byear{1999}).
\bcomment{doi:{\color{blue} \href{https://doi.org/10.1007/s101070050094}
  {10.1007/s101070050094}}}
\end{barticle}
\endbibitem

%%% 5
\bibitem[\protect\citeauthoryear{Peng et~al.}{1999}]{peng1999box}
\begin{barticle}
\bauthor{\bsnm{Peng}, \binits{J.-M.}},
\bauthor{\bsnm{Kanzowb}, \binits{C.}},
\bauthor{\bsnm{Fukushima}, \binits{M.}}:
\batitle{A hybrid josephy newton method for solving box constrained variational
  equality roblems via the d-gap function}.
\bjtitle{Optimization Methods and Software}
\bvolume{10}(\bissue{5}),
\bfpage{687}--\blpage{710}
(\byear{1999}).
\bcomment{doi:{\color{blue} \href{https://doi.org/10.1080/10556789908805734}
  {10.1080/10556789908805734}}}
\end{barticle}
\endbibitem

%%% 6
\bibitem[\protect\citeauthoryear{Ferris et~al.}{1999}]{ferris1999feasible}
\begin{barticle}
\bauthor{\bsnm{Ferris}, \binits{M.C.}},
\bauthor{\bsnm{Kanzow}, \binits{C.}},
\bauthor{\bsnm{Munson}, \binits{T.S.}}:
\batitle{Feasible descent algorithms for mixed complementarity problems}.
\bjtitle{Mathematical Programming}
\bvolume{86}(\bissue{3}),
\bfpage{475}--\blpage{497}
(\byear{1999}).
\bcomment{doi:{\color{blue} \href{https://doi.org/10.1007/s101070050101}
  {10.1007/s101070050101}}}
\end{barticle}
\endbibitem

%%% 7
\bibitem[\protect\citeauthoryear{Solodov and
  Svaiter}{1999a}]{solodov1998globally}
\begin{bbook}
\bauthor{\bsnm{Solodov}, \binits{M.V.}},
\bauthor{\bsnm{Svaiter}, \binits{B.F.}}:
In: \beditor{\bsnm{Fukushima}, \binits{M.}},
\beditor{\bsnm{Qi}, \binits{L.}} (eds.)
\bbtitle{A Globally Convergent Inexact Newton Method for Systems of Monotone
  Equations},
pp. \bfpage{355}--\blpage{369}.
\bpublisher{Springer},
\blocation{Boston, MA}
(\byear{1999})
\end{bbook}
\endbibitem

%%% 8
\bibitem[\protect\citeauthoryear{Solodov and Svaiter}{1999b}]{solodov1999new}
\begin{barticle}
\bauthor{\bsnm{Solodov}, \binits{M.V.}},
\bauthor{\bsnm{Svaiter}, \binits{B.F.}}:
\batitle{A new projection method for variational inequality problems}.
\bjtitle{SIAM Journal on Control and Optimization}
\bvolume{37}(\bissue{3}),
\bfpage{765}--\blpage{776}
(\byear{1999}).
\bcomment{doi:{\color{blue} \href{https://doi.org/10.1137/S0363012997317475}
  {10.1137/S0363012997317475}}}
\end{barticle}
\endbibitem

%%% 9
\bibitem[\protect\citeauthoryear{Solodov and Svaiter}{2000}]{solodov2000truly}
\begin{barticle}
\bauthor{\bsnm{Solodov}, \binits{M.V.}},
\bauthor{\bsnm{Svaiter}, \binits{B.F.}}:
\batitle{A truly globally convergent newton-type method for the monotone
  nonlinear complementarity problem}.
\bjtitle{SIAM Journal on Optimization}
\bvolume{10}(\bissue{2}),
\bfpage{605}--\blpage{625}
(\byear{2000}).
\bcomment{doi:{\color{blue} \href{https://doi.org/10.1137/S1052623498337546}
  {10.1137/S1052623498337546}}}
\end{barticle}
\endbibitem

%%% 10
\bibitem[\protect\citeauthoryear{Han}{2004}]{Han2004}
\begin{barticle}
\bauthor{\bsnm{Han}, \binits{D.}}:
\batitle{A truly globally convergent feasible newton-type method for mixed
  complementarity problems}.
\bjtitle{Journal of Computational Mathematics}
\bvolume{22}(\bissue{3}),
\bfpage{347}--\blpage{360}
(\byear{2004})
\end{barticle}
\endbibitem

%%% 11
\bibitem[\protect\citeauthoryear{Qu et~al.}{2020}]{Qu2020}
\begin{barticle}
\bauthor{\bsnm{Qu}, \binits{B.}},
\bauthor{\bsnm{Wang}, \binits{C.}},
\bauthor{\bsnm{Meng}, \binits{F.}}:
\batitle{Convergence analysis of a projection algorithm for variational
  inequality problems}.
\bjtitle{Journal of Global Optimization}
\bvolume{76},
\bfpage{433}--\blpage{452}
(\byear{2020}).
\bcomment{doi:{\color{blue} \href{https://doi.org/10.1007/s10898-019-00848-0}
  {10.1007/s10898-019-00848-0}}}
\end{barticle}
\endbibitem

%%% 12
\bibitem[\protect\citeauthoryear{Qu et~al.}{2017}]{Qu2017}
\begin{barticle}
\bauthor{\bsnm{Qu}, \binits{B.}},
\bauthor{\bsnm{Wang}, \binits{C.}},
\bauthor{\bsnm{Xiu}, \binits{N.}}:
\batitle{Analysis on newton projection method for the split feasibility
  problem}.
\bjtitle{Computational Optimization and Applications}
\bvolume{67}(\bissue{1}),
\bfpage{175}--\blpage{199}
(\byear{2017}).
\bcomment{doi:{\color{blue} \href{https://doi.org/10.1007/s10589-016-9884-3}
  {10.1007/s10589-016-9884-3}}}
\end{barticle}
\endbibitem

%%% 13
\bibitem[\protect\citeauthoryear{Abdi and Shakeri}{2019}]{abdi2019globally}
\begin{barticle}
\bauthor{\bsnm{Abdi}, \binits{F.}},
\bauthor{\bsnm{Shakeri}, \binits{F.}}:
\batitle{A globally convergent BFGS method for pseudo-monotone variational
  inequality problems}.
\bjtitle{Optimization Methods and Software}
\bvolume{34}(\bissue{1}),
\bfpage{25}--\blpage{36}
(\byear{2019}).
\bcomment{doi:{\color{blue}
  \href{https://doi.org/10.1080/10556788.2017.1332619}
  {10.1080/10556788.2017.1332619}}}
\end{barticle}
\endbibitem

%%% 14
\bibitem[\protect\citeauthoryear{Abdi and Shakeri}{2017}]{abdi2017new}
\begin{barticle}
\bauthor{\bsnm{Abdi}, \binits{F.}},
\bauthor{\bsnm{Shakeri}, \binits{F.}}:
\batitle{A new descent method for symmetric non-monotone variational
  inequalities with application to eigenvalue complementarity problems}.
\bjtitle{Journal of Optimization Theory and Applications}
\bvolume{173}(\bissue{3}),
\bfpage{923}--\blpage{940}
(\byear{2017}).
\bcomment{doi:{\color{blue} \href{https://doi.org/10.1007/s10957-017-1100-9}
  {10.1007/s10957-017-1100-9}}}
\end{barticle}
\endbibitem

%%% 15
\bibitem[\protect\citeauthoryear{Bertsekas and
  Tsitsiklis}{2015}]{Bertsekas1989}
\begin{bbook}
\bauthor{\bsnm{Bertsekas}, \binits{D.}},
\bauthor{\bsnm{Tsitsiklis}, \binits{J.}}:
\bbtitle{Parallel and Distributed Computation: Numerical Methods}.
\bpublisher{Athena Scientific}
(\byear{2015}).
\burl{https://books.google.de/books?id=n_Q5EAAAQBAJ}
\end{bbook}
\endbibitem

%%% 16
\bibitem[\protect\citeauthoryear{Fukushima}{1992}]{Fukushima1992}
\begin{barticle}
\bauthor{\bsnm{Fukushima}, \binits{M.}}:
\batitle{Equivalent differentiable optimization problems and descent methods
  for asymmetric variational inequality problems}.
\bjtitle{Mathematical Programming}
\bvolume{53},
\bfpage{99}--\blpage{110}
(\byear{1992}).
\bcomment{doi:{\color{blue} \href{https://doi.org/10.1007/BF01585696}
  {10.1007/BF01585696}}}
\end{barticle}
\endbibitem

%%% 17
\bibitem[\protect\citeauthoryear{Li and Fukushima}{2001}]{2000LiDH}
\begin{barticle}
\bauthor{\bsnm{Li}, \binits{D.-H.}},
\bauthor{\bsnm{Fukushima}, \binits{M.}}:
\batitle{On the global convergence of the BFGS method for nonconvex
  unconstrained optimization problems}.
\bjtitle{SIAM Journal on Optimization}
\bvolume{11}(\bissue{4}),
\bfpage{1054}--\blpage{1064}
(\byear{2001}).
\bcomment{doi:{\color{blue} \href{https://doi.org/10.1137/S1052623499354242}
  {10.1137/S1052623499354242}}}
\end{barticle}
\endbibitem

%%% 18
\bibitem[\protect\citeauthoryear{Rockafellar}{1969}]{rockafellar1969convex}
\begin{bchapter}
\bauthor{\bsnm{Rockafellar}, \binits{R.T.}}:
\bctitle{Convex functions, monotone operators and variational inequalities}.
In: \bbtitle{Theory and Applications of Monotone Operators},
pp. \bfpage{35}--\blpage{65}
(\byear{1969}).
\bcomment{Citeseer}
\end{bchapter}
\endbibitem

%%% 19
\bibitem[\protect\citeauthoryear{Pang}{1987}]{doi:10.1287/moor.12.3.474}
\begin{barticle}
\bauthor{\bsnm{Pang}, \binits{J.-S.}}:
\batitle{A posteriori error bounds for the linearly-constrained variational
  inequality problem}.
\bjtitle{Mathematics of Operations Research}
\bvolume{12}(\bissue{3}),
\bfpage{474}--\blpage{484}
(\byear{1987}).
\bcomment{doi:{\color{blue} \href{https://doi.org/10.1287/moor.12.3.474}
  {10.1287/moor.12.3.474}}}
\end{barticle}
\endbibitem

%%% 20
\bibitem[\protect\citeauthoryear{Taji}{2008}]{Taji2008}
\begin{barticle}
\bauthor{\bsnm{Taji}, \binits{K.}}:
\batitle{A note on globally convergent newton method for strongly monotone
  variational inequalities}.
\bjtitle{Journal of the Operations Research Society of Japan}
\bvolume{51}(\bissue{4}),
\bfpage{310}--\blpage{316}
(\byear{2008}).
\bcomment{doi:{\color{blue} \href{https://doi.org/10.15807/jorsj.51.310}
  {10.15807/jorsj.51.310}}}
\end{barticle}
\endbibitem

%%% 21
\bibitem[\protect\citeauthoryear{}{2012}]{HanBBVI}
\begin{botherref}
Some projection methods with the BB step sizes for variational inequalities.
Journal of Computational and Applied Mathematics
\textbf{236}(9),
2590--2604
(2012).
doi:{\color{blue} \href{https://doi.org/10.1016/j.cam.2011.12.017}
  {10.1016/j.cam.2011.12.017}}
\end{botherref}
\endbibitem

%%% 22
\bibitem[\protect\citeauthoryear{Zhou and Li}{2008}]{Zhou2008}
\begin{barticle}
\bauthor{\bsnm{Zhou}, \binits{W.-J.}},
\bauthor{\bsnm{Li}, \binits{D.-H.}}:
\batitle{A globally convergent BFGS method for nonlinear monotone equations
  without any merit functions}.
\bjtitle{Mathematics Of Computation}
\bvolume{77}(\bissue{264}),
\bfpage{2231}--\blpage{2240}
(\byear{2008}).
\bcomment{doi:{\color{blue}
  \href{https://doi.org/10.1090/S0025-5718-08-02121-2}
  {10.1090/S0025-5718-08-02121-2}}}
\end{barticle}
\endbibitem

%%% 23
\bibitem[\protect\citeauthoryear{Kebaili and Benterki}{2018}]{Kebaili2018A}
\begin{barticle}
\bauthor{\bsnm{Kebaili}, \binits{Z.}},
\bauthor{\bsnm{Benterki}, \binits{D.}}:
\batitle{A penalty approach for a box constrained variational inequality
  problem}.
\bjtitle{Applications of Mathematics}
\bvolume{63}(\bissue{4}),
\bfpage{439}--\blpage{454}
(\byear{2018}).
\bcomment{doi:{\color{blue} \href{https://doi.org/10.21136/AM.2018.0334-17}
  {10.21136/AM.2018.0334-17}}}
\end{barticle}
\endbibitem

%%% 24
\bibitem[\protect\citeauthoryear{Harker and Pang}{1990}]{Harker1990A}
\begin{botherref}
\oauthor{\bsnm{Harker}, \binits{P.T.}},
\oauthor{\bsnm{Pang}, \binits{J.-S.}}:
A damped-{Newton} method for the linear complementarity problem.
Computational solution of nonlinear systems of equations, {Proc}. {SIAM}-{AMS}
  {Summer} {Semin}., {Ft}. {Collins}/{CO} ({USA}) 1988, {Lect}. {Appl}. {Math}.
  26, 265-284 (1990).
(1990)
\end{botherref}
\endbibitem

%%% 25
\bibitem[\protect\citeauthoryear{Majig and
  Fukushima}{2008}]{majig2008restricted}
\begin{botherref}
\oauthor{\bsnm{Majig}, \binits{M.-A.}},
\oauthor{\bsnm{Fukushima}, \binits{M.}}:
Restricted-step josephy-newton method for general variational inequalities with
  polyhedral constraints.
Report, Kyoto University
(2008)
\end{botherref}
\endbibitem

\end{thebibliography}
%% if required, the content of .bbl file can be included here once bbl is generated
%%\input sn-article.bbl

\end{document}